\documentclass[12pt,a4paper,reqno]{amsart}
\usepackage[margin=1in]{geometry}
\usepackage{amsmath,amssymb,amsthm}
\usepackage{hyperref}
\hypersetup{colorlinks=true,linkcolor=blue,citecolor=blue,urlcolor=blue}

\numberwithin{equation}{section}

\newtheorem{theorem}{Theorem}[section]
\newtheorem{proposition}[theorem]{Proposition}
\newtheorem{lemma}[theorem]{Lemma}

\theoremstyle{remark}
\newtheorem{remark}[theorem]{Remark}
\newcommand{\C}{\mathbb C}
\DeclareMathOperator{\Hol}{Hol}
\DeclareMathOperator{\Sym}{Sym}

\begin{document}

\title{Commuting Toeplitz operators with harmonic symbols on the Fock space}
\date{}
\author[Jie Qin]{Jie Qin$^{1,*}$}
\thanks{The author was supported by the National Natural Science Foundation of China (12501154) and the Science and Technology Research Program of Chongqing Municipal Education Commission (KJQN202400806).}
\author[Youqi Liu]{Youqi Liu$^1$}
\address{$^1$ School of Mathematics and Statistics,
Chongqing Technology and Business University,
Chongqing 400067, China}
\email{qinjie24520@163.com; yqliu2021@163.com}
\author[Sui Huang]{Sui Huang$^2$}
\address{$^2$ School of Mathematical Sciences,
Chongqing Normal University, 401331 Chongqing, China}
\email{huangsui2@163.com}

\begin{abstract}
We give a counterexample to a characterization of commuting
Toeplitz operators due to Bauer, Choe, and Koo [JFA, 268 (2015), 3017--3060].
We prove a corrected characterization for Toeplitz operators with
harmonic symbols of exponential type on the Fock space in one complex
variable.
\end{abstract}

\thanks{$^*$Corresponding author.}

\maketitle

\section{Introduction}

We use the notation and normalization of \cite{Bauer2015}. Let
$$
d\mu(z)=\pi^{-1}e^{-|z|^2}dA(z),
\qquad
H^2(\C)=L^2(\C,d\mu)\cap\Hol(\C),
$$
where $dA$ is Lebesgue area measure and $\Hol(\C)$ denotes the space of
entire functions. All linear dependence is over $\C$. The reproducing
kernel at $w$ and the normalized reproducing kernel are
$$
K_w(z)=e^{z\overline w},
\qquad
k_w(z)=e^{z\overline w-|w|^2/2}.
$$

For $c>0$, let $D_c$ be the class of measurable functions $u$ such
that $ue^{-c|\cdot|^2}$ is essentially bounded. Set
$$
\Sym(\C)=\bigcap_{c>0}D_c,
\qquad
H(\C)=\bigg(\bigcup_{0<c<1/2}D_c\bigg)\cap\Hol(\C).
$$
The class $\mathcal E(\C)$ consists of the measurable functions satisfying
$|u(z)|\le Ce^{s|z|}$ almost everywhere for some $C,s>0$.
This bound gives $\mathcal E(\C)\subset\Sym(\C)$. For $u\in\Sym(\C)$,
the Toeplitz operator with symbol $u$ is $T_u=PM_u$, where $P$ is
the orthogonal projection from $L^2(\C,d\mu)$ onto $H^2(\C)$ and
$M_u$ is multiplication by $u$. The space
$H(\C)$ is a common invariant dense domain for these operators by
\cite[Section 2.2]{Bauer2015}. All operator identities below are understood
on this domain. Following \cite[Section 2.5]{Bauer2015}, we write
$$
(T_u,T_v]:=T_{uv}-T_uT_v,
\qquad
[T_u,T_v]:=T_uT_v-T_vT_u.
$$

The Berezin transform of a symbol $u\in\Sym(\C)$ is
$$
\mathcal B[u](z)=\int_{\C}u(z\pm w)\,d\mu(w).
$$
Denote by $\mathcal T$ the algebra of finite sums of finite products of
Toeplitz operators with symbols in $\Sym(\C)$. For $T\in\mathcal T$, set
$$
\mathcal B[T](w)=\langle Tk_w,k_w\rangle\quad\text{and}\quad
\langle u,v\rangle=\int_{\C}u\overline v\,d\mu.
$$

Bauer, Choe, and Koo \cite{Bauer2015} studied commuting Toeplitz operators
with pluriharmonic symbols on the Fock space. For harmonic symbols of
exponential type in one variable, they gave four alternatives
for commutativity. The example in Theorem \ref{thm:counterexample}
satisfies none of these alternatives.

The argument in Case 2 of the proof of
\cite[Proposition 4.5]{Bauer2015} does not account for cancellation
between terms with the same exponential parameters.
We replace this proposition with a criterion expressed in terms of the
polynomial coefficients. The corrected characterization follows from this
criterion and the results in \cite[Section 4]{Bauer2015}.
The equivalence between commutativity and invariance under the
Berezin transform remains valid. We recall this equivalence and
the corresponding result for semi-commutators in one complex variable.

\begin{lemma}[{\cite[Proposition 2.7]{Bauer2015}}]
\label{lem:berezin-injectivity}
Let $f,g,h,k\in\Sym(\C)\cap\Hol(\C)$. The following
equivalences hold.
\begin{enumerate}
\item[(a)]
$
(T_f,T_{\overline g}]=0
\quad\Longleftrightarrow\quad
\mathcal B[f\overline g]=f\overline g.
$
\item[(b)]
$
[T_{f+\overline k},T_{h+\overline g}]=0
\quad\Longleftrightarrow\quad
\mathcal B[f\overline g-h\overline k]
=f\overline g-h\overline k.
$
\end{enumerate}
\end{lemma}

Part (b) gives the equivalence of (a) and (b) in the following
theorem, which we quote as originally stated.
The counterexample in Theorem \ref{thm:counterexample} shows
that (b) does not imply (c).
\begin{theorem}[{\cite[Theorem 4.15]{Bauer2015}}]\label{thm:bauer-characterization}
Suppose $f,g,h,k\in\Hol(\C)\cap\mathcal E(\C)$. The following statements
are equivalent.
\begin{enumerate}
\item[(a)] $[T_{f+\overline k},T_{h+\overline g}]=0$.
\item[(b)] $\mathcal B[f\overline g-h\overline k]
=f\overline g-h\overline k$.
\item[(c)] One of the following four conditions holds.
\begin{enumerate}
\item[(i)] $(T_f,T_{\overline g}]=(T_h,T_{\overline k}]=0$.
\item[(ii)] $\{f+\overline k,h+\overline g,1\}$ is linearly dependent.
\item[(iii)] There are constants $a,b,c\in\C$ with $ab\ne0$ such that
$$
af+bh=c
\quad\text{and}\quad
\mathcal B[f(b\overline g+a\overline k)]
=f(b\overline g+a\overline k).
$$
\item[(iv)] There are constants $a,b,c\in\C$ with $ab\ne0$ such that
$$
ag+bk=c
\quad\text{and}\quad
\mathcal B[(\overline b f+\overline a h)\overline g]
=(\overline b f+\overline a h)\overline g.
$$
\end{enumerate}
\end{enumerate}
\end{theorem}

\section{The counterexample}

Define
\begin{equation}\label{eq:symbols}
\begin{aligned}
f(z)&=e^z, & h(z)&=e^z+e^{2z},\\
k(z)&=e^{-\pi iz}, & g(z)&=e^{-\pi iz}+e^{-2\pi iz}.
\end{aligned}
\end{equation}
These functions belong to $\Hol(\C)\cap\mathcal E(\C)$, since
$|e^{az}|\le e^{|a||z|}$. The following lemma compares the two semi-commutators in the first
alternative of the original characterization.

\begin{lemma}\label{lem:nonzero-semi}
For the functions in \eqref{eq:symbols}, we have
\begin{equation}\label{eq:nonzero-semi}
(T_f,T_{\overline g}]=(T_h,T_{\overline k}]\ne0.
\end{equation}
\end{lemma}

\begin{proof}
For $a,b,c\in\C$, the reproducing property gives
\begin{align*}
T_{e^{bz+c\overline z}}e^{az}
&=\int_{\C}e^{(a+b)w+(z+c)\overline w}\,d\mu(w)\\
&=e^{(a+b)(z+c)}.
\end{align*}
Applying this formula to \eqref{eq:symbols}, we obtain
$$
\begin{aligned}
T_{f\overline g}e^{az}
&=(-e^{\pi ia}+e^{2\pi ia})e^{(a+1)z},\\
T_fT_{\overline g}e^{az}
&=(e^{\pi ia}+e^{2\pi ia})e^{(a+1)z},\\
T_{h\overline k}e^{az}
&=-e^{\pi ia}e^{(a+1)z}+e^{\pi ia}e^{(a+2)z},\\
T_hT_{\overline k}e^{az}
&=e^{\pi ia}e^{(a+1)z}+e^{\pi ia}e^{(a+2)z}.
\end{aligned}
$$
Subtracting the corresponding expressions gives
\begin{equation}\label{eq:semi-on-exponentials}
(T_f,T_{\overline g}](e^{az})
=(T_h,T_{\overline k}](e^{az})
=-2e^{\pi ia}e^{(a+1)z}.
\end{equation}
Since $e^{az}=K_{\overline a}$, the two semi-commutators agree on
every normalized kernel and have the same Berezin transform.
It follows that
$$
\mathcal B[f\overline g]-f\overline g
=\mathcal B[h\overline k]-h\overline k,
$$
which is equivalent by linearity to
$$
\mathcal B[f\overline g-h\overline k]
=f\overline g-h\overline k.
$$
The four functions belong to $\Hol(\C)\cap\mathcal E(\C)
\subset\Hol(\C)\cap\Sym(\C)$. Lemma
\ref{lem:berezin-injectivity} (b), together with the commutator
decomposition in \cite[Section 2.5]{Bauer2015}, now yields
$$
0=[T_{f+\overline k},T_{h+\overline g}]
=(T_h,T_{\overline k}]-(T_f,T_{\overline g}].
$$
Taking $a=0$ in \eqref{eq:semi-on-exponentials} gives
$$
(T_f,T_{\overline g}]1
=(T_h,T_{\overline k}]1=-2e^z\ne0,
$$
so the common semi-commutator is nonzero. The proof is complete.
\end{proof}

We next compute the Berezin transforms of the two mixed products
explicitly and record the resulting commuting pair.
\begin{lemma}\label{lem:commuting}
For the functions in \eqref{eq:symbols}, we have
\begin{equation}\label{eq:commuting}
[T_{f+\overline k},T_{h+\overline g}]=0.
\end{equation}
Moreover, we have
$$
(T_f,T_{\overline{g}}]=(T_h,T_{\overline{k}}]\neq0.
$$
\end{lemma}

\begin{proof}
Taking kernel parameters $\overline a,b$ and constant polynomial
factors in \cite[Lemma 3.3]{Bauer2015}, we obtain
$$
\mathcal B[e^{az+b\overline z}]
=e^{ab}e^{az+b\overline z},\qquad a,b\in\C.
$$
The expansions in \eqref{eq:symbols} give
$$
\begin{aligned}
f\overline g
&=e^{z+\pi i\overline z}+e^{z+2\pi i\overline z}\quad \text{and}\quad
h\overline k
&=e^{z+\pi i\overline z}+e^{2z+\pi i\overline z}.
\end{aligned}
$$
The preceding Berezin formula multiplies the common term by
$e^{\pi i}=-1$ and each remaining term by $e^{2\pi i}=1$, which gives
\begin{equation}\label{eq:equal-defects}
\mathcal B[f\overline g]-f\overline g
=\mathcal B[h\overline k]-h\overline k
=-2e^{z+\pi i\overline z}\ne0.
\end{equation}
Subtracting these identities gives
$$
\mathcal B[f\overline g-h\overline k]
=f\overline g-h\overline k,
$$
and Lemma \ref{lem:berezin-injectivity} (b) proves
\eqref{eq:commuting}. The assertion about the semi-commutators
is \eqref{eq:nonzero-semi}. This completes the proof.
\end{proof}

It remains to show that this commuting pair satisfies none
of the original four alternatives.
\begin{theorem}\label{thm:counterexample}
The functions $f,g,h,k$ in \eqref{eq:symbols} satisfy (a) and (b) of
Theorem \ref{thm:bauer-characterization}, but none of the alternatives
in (c). More precisely, each of the families
\begin{equation}\label{eq:independent-families}
\{f,h,1\},\qquad
\{g,k,1\},\qquad
\{f+\overline k,h+\overline g,1\}
\end{equation}
is linearly independent.
\end{theorem}

\begin{proof}
Lemma \ref{lem:commuting} gives (a), and subtracting the identities
in \eqref{eq:equal-defects} gives (b). Alternative (c)(i) is
excluded by \eqref{eq:nonzero-semi}.
By \eqref{eq:symbols}, the independence of $1,e^z,e^{2z}$ and of
$1,e^{-\pi iz},e^{-2\pi iz}$ gives the independence of
$\{f,h,1\}$ and $\{g,k,1\}$, respectively.
To treat the third family, suppose that
$$
a(f+\overline k)+b(h+\overline g)+c=0.
$$
Differentiation with respect to $z$ eliminates the antiholomorphic
terms and gives
$$
(a+b)e^z+2be^{2z}=0.
$$
The independence of $e^z,e^{2z}$ implies $a+b=2b=0$.
Thus $a=b=0$, and the original relation gives $c=0$.
The first two independence assertions in
\eqref{eq:independent-families} exclude (c)(iii) and (c)(iv),
while the third excludes (c)(ii). The proof is complete.
\end{proof}

Following \cite{Bauer2015}, let $\mathcal P_1$ denote the algebra of
holomorphic polynomials in one variable. Let $\mathcal A_1$ be the
algebra generated by $\mathcal P_1$ and the reproducing kernels. Thus
$$
\mathcal A_1
=\left\{\sum_{j=0}^{J}p_jK_{a_j}:
J\ge0,\ p_j\in\mathcal P_1,\ a_j\in\C\right\}.
$$
After collecting terms with the same kernel parameter, each
$u\in\mathcal A_1$ has a unique decomposition
$$
 u=P_u+E_u,\quad
 E_u=\sum_{j=1}^{N}p_jK_{a_j},
$$
where $P_u\in\mathcal P_1$, the parameters $a_j$ are distinct
and nonzero, and the polynomials $p_j$ are nonzero. The polynomial
part $P_u$ is the term with kernel parameter zero, taken to be zero
if that parameter is absent. The remaining sum $E_u$ is the
nonpolynomial part of $u$, and it is empty when $u$ is a polynomial.
This is the decomposition used in the proof of
\cite[Proposition 4.5]{Bauer2015}.

In Case 2 of that proof, both pairs
$\{E_f,E_h\}$ and $\{E_g,E_k\}$ are assumed to be linearly
independent over $\mathbb C$. From
$$
 \mathcal B[E_f\overline{E_g}+E_h\overline{E_k}]
 =E_f\overline{E_g}+E_h\overline{E_k},
$$
the authors infer that
$$
 \mathcal B[E_f\overline{E_g}]=E_f\overline{E_g}\quad\text{and}\quad
 \mathcal B[E_h\overline{E_k}]=E_h\overline{E_k}.
$$
For the four functions $(f,g,-h,k)$ from \eqref{eq:symbols},
all polynomial parts vanish and both pairs of nonpolynomial parts
are linearly independent. However, \eqref{eq:equal-defects} shows
that neither summand is fixed by $\mathcal B$, so this inference fails.
The term $e^{z+\pi i\overline z}$ occurs in both
$f\overline g$ and $h\overline k$ and cancels in their
difference. Linear independence of the two pairs does not
exclude this overlap.

The statement of \cite[Proposition 4.5]{Bauer2015} also fails
for these functions. Alternative (i) is excluded by
\eqref{eq:equal-defects}. To exclude alternative (ii), set
$$
 u=f\overline g-h\overline k
 =e^{z+2\pi i\overline z}-e^{2z+\pi i\overline z}.
$$
Suppose that $u-A\overline B$ is harmonic for some entire
functions $A,B$, and define
$$
 B^*(w)=\overline{B(\overline w)}.
$$
Applying $\partial_z\partial_{\overline z}$ to the harmonic remainder gives
$$
\partial_z\partial_{\overline z}u
=2\pi i u=A'(z)\overline{B'(z)}.
$$
Both sides extend to entire functions of $(z,w)$, so
\cite[Lemma 2.9]{Bauer2015} gives
$$
 2\pi i\bigl(e^{z+2\pi iw}-e^{2z+\pi iw}\bigr)
 =A'(z)(B^*)'(w).
$$
Taking $w=0$ and $w=1$, respectively, yields
$$
\begin{aligned}
2\pi i(e^z-e^{2z})&=A'(z)(B^*)'(0)\quad\text{and}\quad
2\pi i(e^z+e^{2z})&=A'(z)(B^*)'(1).
\end{aligned}
$$
The two functions on the left are linearly independent, whereas
both functions on the right are scalar multiples of $A'$. This contradiction excludes such a
representation even without any growth restriction on $A,B$.
It therefore also excludes alternative (ii) of
\cite[Corollary 4.11]{Bauer2015}. Alternative (i) again fails
by \eqref{eq:equal-defects}.

The same example also contradicts \cite[Theorem 4.14]{Bauer2015}.
Its first alternative fails by \eqref{eq:equal-defects},
and its second fails because
$\partial_z\partial_{\overline z}u=2\pi i u\ne0$.
The remaining alternatives require a nontrivial affine
relation within $\{f,h\}$ or $\{g,k\}$.
Since all polynomial parts vanish, such a relation would
give a linear dependence of the corresponding nonpolynomial
parts, contrary to their independence.

\section{Differences of two products of exponential polynomials}

In this section, we replace \cite[Proposition 4.5]{Bauer2015}
with a criterion expressed in the polynomial coefficients of
$f\overline g-h\overline k$, where $f,g,h,k$ are arbitrary functions
in $\mathcal A_1$. The corresponding criterion for
$f\overline g+h\overline k$ in Bauer's equation follows by
replacing $h$ with $-h$.
As in \cite[Section 3]{Bauer2015}, for $a,b\in\C$ and
$q\in\mathcal P_1$, put
$$
K_{a,b}=K_a\overline{K_b}\quad\text{and}\quad
q^*(w)=\overline{q(\overline w)}.
$$
So $K_{a,b}(z)=e^{\overline a z+b\overline z}$ and
$q^*(\overline z)=\overline{q(z)}$.
We replace $\overline z$ by an independent complex variable $w$
when comparing coefficients and taking partial derivatives.
We write $\C[z,w]$ for the polynomials in these two variables.
On $\C[z,w]$, the derivatives $\partial_z$ and $\partial_w$ act on
the independent variables $z$ and $w$.
Setting $w=\overline z$ recovers the original function.
For $R\in\C[z,w]$, we denote this restriction by
$R(\cdot,\overline{\cdot})$, namely $z\mapsto R(z,\overline z)$.
The chain rule gives
\begin{equation}\label{E13}
\partial_{\overline z}[R(z,\overline z)]=(\partial_wR)(z,\overline z)
\quad\text{and}\quad\partial_z\partial_{\overline z}[R(z,\overline z)]= (\partial_z\partial_wR)(z,\overline z).
\end{equation}
Note that derivatives on the right are taken in the independent variables
$z,w$ before setting $w=\overline z$, and those on the left
act on the restricted function $z\mapsto R(z,\overline z)$.

For $f,g,h,k\in\mathcal A_1$, write
\begin{equation}\label{eq:separate-expansions}
\begin{aligned}
f&=\sum_{j=0}^{J}p_jK_{a_j},&
h&=\sum_{m=0}^{M}r_mK_{c_m},\\
g&=\sum_{\ell=0}^{L}q_\ell K_{b_\ell},&
k&=\sum_{n=0}^{N}s_nK_{d_n},
\end{aligned}
\end{equation}
where all coefficients belong to $\mathcal P_1$. We choose the four lists of kernel parameters separately, requiring distinct entries within each list but allowing the same parameter to occur in different lists. We take
$a_0=b_0=c_0=d_0=0$ and allow the corresponding coefficients to vanish.
The polynomial parts are $P_f=p_0$, $P_h=r_0$, $P_g=q_0$,
and $P_k=s_0$. The remaining terms form $E_f,E_h,E_g,E_k$,
respectively, with polynomial coefficients that need not be constant.

The parameter pairs in the two products form the finite set
$$
\begin{aligned}
\Lambda={}&\{(a_j,b_\ell):0\le j\le J,\ 0\le\ell\le L\}\\
&{}\cup\{(c_m,d_n):0\le m\le M,\ 0\le n\le N\}.
\end{aligned}
$$
For $(a,b)\in\Lambda$, we combine the coefficients at this pair and set
\begin{equation}\label{eq:merged-coefficients}
C_{a,b}(z,w)
=\sum_{\substack{a_j=a\\b_\ell=b}}p_j(z)q_\ell^*(w)
-\sum_{\substack{c_m=a\\d_n=b}}r_m(z)s_n^*(w).
\end{equation}
Empty sums are understood to be zero. The parameters in each list are distinct, so each sum contains
at most one term. Expanding the two products, collecting terms with the same
parameter pair, and setting $w=\overline z$ gives
\begin{equation}\label{eq:merged-expansion}
f\overline g- h\overline k
=\sum_{(a,b)\in\Lambda}
C_{a,b}(z,\overline z)K_{a,b}(z).
\end{equation}
We say that a family $C=(C_{a,b})_{(a,b)\in\Lambda}$ satisfies
condition \eqref{eq:coefficient-condition} if every
$(a,b)\in\Lambda$ with $C_{a,b}\not\equiv0$ satisfies one of
the following conditions:
\begin{equation*}\label{eq:coefficient-condition}\tag{C}
\begin{aligned}
\text{(i)}\quad
& ab=0,\quad \overline a b\in2\pi i\mathbb Z, \quad C_{a,b}(z,\overline z)K_{a,b}(z)
  \text{ is harmonic on }\C.\\[6pt]
\text{(ii)}\quad
& ab\ne0,\quad \overline a b\in2\pi i\mathbb Z, \quad C_{a,b}(z,\overline z)
  =c_{a,b}+d_{a,b}(\overline a z-b\overline z).
\end{aligned}
\end{equation*}
In (ii), there are constants $c_{a,b},d_{a,b}\in\C$, depending
only on $(a,b)$.
Here $\mathbb Z$ denotes the set of integers, and
$$
 2\pi i\mathbb Z=\{2\pi i k:k\in\mathbb Z\}.
$$
By \cite[Lemma 2.9]{Bauer2015}, the polynomial identity in the
second case is equivalent to the one obtained by replacing
$\overline z$ with an independent variable $w$.

The following proposition characterizes invariance under the
Berezin transform in terms of these combined coefficients $C$.

\begin{proposition}[Replacement for {\cite[Proposition 4.5]{Bauer2015}}]
\label{prop:coefficient-criterion}
Let $f,g,h,k\in\mathcal A_1$ have the representations
\eqref{eq:separate-expansions}. Then
$$
\mathcal B[f\overline g- h\overline k]
=f\overline g- h\overline k
$$
if and only if the family $C=(C_{a,b})_{(a,b)\in\Lambda}$
in \eqref{eq:merged-coefficients} satisfies
\eqref{eq:coefficient-condition}.
\end{proposition}

\begin{proof}
Set $F=f\overline g-h\overline k$. For $a,b\in\C$, put
\begin{equation}\label{E1}
D_{a,b}=\partial_z\partial_w+b\partial_z
+\overline a\partial_w.
\end{equation}
The operator $D_{a,b}$ acts on $\C[z,w]$. For each integer $m\ge0$,
set
$$
V_m=
\bigg\{
\sum_{\substack{r,s\ge0\\r+s\le m}}c_{rs}z^rw^s:
c_{rs}\in\C
\bigg\}.
$$
Then $V_m$ consists of the polynomials of total degree at most $m$,
including the zero polynomial. The total degree of $z^rw^s$
is $r+s$. Let $I$ denote the identity operator on $V_m$, so
$IR=R$ for every $R\in V_m$. For $\nu\ge1$, the notation
$D_{a,b}^{\nu}$ means the composition of $\nu$ copies of $D_{a,b}$,
and $D_{a,b}^{0}=I$.

For a polynomial $R$ of total degree $m$, each nonzero first partial derivative has total degree at most $m-1$. Its mixed derivative $\partial_z\partial_wR$ has total degree at most $m-2$ when nonzero. So we obtain $D_{a,b}(V_0)=\{0\}$ and
$$
D_{a,b}(V_j)\subseteq V_{j-1}
$$
for $j\ge1$.
For $R\in V_m$, repeated application yields
$$
D_{a,b}^{\nu}R\in V_{m-\nu}
$$
for $0\le\nu\le m.$
In particular, $D_{a,b}^{m}R$ is constant. Applying $D_{a,b}$ once more, we obtain
\begin{equation}\label{E2}
D_{a,b}^{m+1}R=0
\end{equation}
for $R\in V_m.$
Thus, $D_{a,b}^{m+1}$ is the zero operator on $V_m$.
This also covers $m=0$, since $D_{a,b}$ annihilates constants.
On $V_m$, the operator exponential is therefore the finite sum
\begin{equation}\label{E3}
e^{D_{a,b}}R
:=\sum_{\nu=0}^{\infty}\frac{D_{a,b}^{\nu}R}{\nu!}
=\sum_{\nu=0}^{m}\frac{D_{a,b}^{\nu}R}{\nu!},
\qquad R\in V_m.
\end{equation}

For $S\in V_m$, Taylor's formula gives
$$
\bigl(e^{b\partial_z}S\bigr)(z,w)
=\sum_{j=0}^{m}\frac{b^j}{j!}(\partial_z^jS)(z,w)
=S(z+b,w).
$$
Similarly, in the variable $w$,
$e^{\overline a\,\partial_w}$ translates $w$ by $\overline a$.
For $R\in V_m$, the corresponding finite expansion for the mixed
derivative is
$$
e^{\partial_z\partial_w}R
=\sum_{\nu=0}^{m}\frac{(\partial_z\partial_w)^\nu R}{\nu!}
=\sum_{\nu=0}^{m}\frac{\partial_z^\nu\partial_w^\nu R}{\nu!}.
$$
The three summands of $D_{a,b}$ commute, so multiplying these
finite expansions gives
\begin{align}\label{EEE2}
\bigl(e^{D_{a,b}}R\bigr)(z,w)
&=\bigl(e^{b\partial_z}e^{\overline a\,\partial_w}
e^{\partial_z\partial_w}R\bigr)(z,w)\nonumber\\
&=\sum_{\nu=0}^{m}\frac{1}{\nu!}
\bigl(\partial_z^\nu\partial_w^\nu R\bigr)
(z+b,w+\overline a).
\end{align}

For $R(z,w)=p(z)q^*(w)$ with $p,q\in\mathcal P_1$,
the sum in \eqref{EEE2} is the polynomial factor in
\cite[Lemma 3.3]{Bauer2015}. For a general polynomial
$R(z,w)=\sum_{r+s\le m}\kappa_{rs}z^rw^s\in V_m$, write
$$
H(z)=R(z,\overline z)K_{a,b}(z)
=\sum_{r+s\le m}\kappa_{rs}(z^rK_a(z))\overline{z^sK_b(z)}.
$$
By linearity of the Berezin transform, we have
$$
\mathcal B[H](z)=\sum_{r+s\le m}\kappa_{rs}
\mathcal B\bigl[(\cdot)^rK_a\,\overline{(\cdot)^sK_b}\bigr](z).
$$
For each summand, \cite[Lemma 3.3]{Bauer2015}, with $p(z)=z^r$ and $q(z)=z^s$, gives
\begin{align*}
\mathcal B\bigl[(\cdot)^rK_a\,\overline{(\cdot)^sK_b}\bigr](z)
={}&e^{\overline a b}K_{a,b}(z)\\
&{}\times\left[(\partial_z+w+\overline a)^s(z+b)^r\right]_{w=\overline z}.
\end{align*}
For $r+s\le m$, the operator $\partial_z$ commutes with
multiplication by $w+\overline a$, so the binomial formula yields
\begin{align*}
(\partial_z+w+\overline a)^s(z+b)^r
&=\sum_{\nu=0}^{s}\binom{s}{\nu}(w+\overline a)^{s-\nu}
\partial_z^\nu(z+b)^r\\[3pt]
&=\sum_{\nu=0}^{\min(r,s)}\frac{1}{\nu!}
\frac{r!}{(r-\nu)!}\frac{s!}{(s-\nu)!}
(z+b)^{r-\nu}(w+\overline a)^{s-\nu}\\[3pt]
&=\sum_{\nu=0}^{m}\frac{
\bigl[\partial_z^\nu\partial_w^\nu(z^rw^s)\bigr]
(z+b,w+\overline a)}{\nu!}.
\end{align*}
The last equality uses the vanishing of the derivatives for
$\nu>\min(r,s)$.
We substitute this expansion into the formula for $\mathcal B[H]$
and interchange the finite sums to obtain
\begin{align*}
\mathcal B[R(\cdot,\overline{\cdot})K_{a,b}](z)
={}&e^{\overline a b}K_{a,b}(z)\\[-2pt]
&{}\times\sum_{\nu=0}^{m}\frac{(\partial_z^\nu\partial_w^\nu R)
(z+b,\overline z+\overline a)}{\nu!}.
\end{align*}
Using \eqref{EEE2}, we can write this identity as
\begin{equation}\label{eq:polynomial-block-transform}
\mathcal B[R(\cdot,\overline{\cdot})K_{a,b}](z)
=e^{\overline a b}K_{a,b}(z)
\bigl(e^{D_{a,b}}R\bigr)(z,\overline z).
\end{equation}

For sufficiency, assume \eqref{eq:coefficient-condition} holds.
The terms in \eqref{eq:merged-expansion} with $ab=0$ are harmonic
and hence fixed by $\mathcal B$.
For $ab\ne0$ and $C_{a,b}\ne0$, complexification of the identity
in \eqref{eq:coefficient-condition} gives
$$
C_{a,b}(z,w)=c_{a,b}+d_{a,b}(\overline a z-bw).
$$
Condition \eqref{eq:coefficient-condition} also implies
$e^{\overline a b}=1$. Substituting the expression for $C_{a,b}$
into \eqref{E1}, we obtain
$$
D_{a,b}C_{a,b}=bd_{a,b}\overline a-\overline a d_{a,b}b=0.
$$
Thus \eqref{E3} reduces to $e^{D_{a,b}}C_{a,b}=C_{a,b}$, and
\eqref{eq:polynomial-block-transform} shows that the corresponding
term in \eqref{eq:merged-expansion} is also fixed.
The terms with zero coefficients vanish, so linearity and
\eqref{eq:merged-expansion} give $\mathcal B[F]=F$.

For necessity, assume $\mathcal B[F]=F$. We apply $\mathcal B$ to
\eqref{eq:merged-expansion} and use
\eqref{eq:polynomial-block-transform} to obtain an identity on
$w=\overline z$. Each term extends to an entire function of
$(z,w)$, so \cite[Lemma 2.9]{Bauer2015} gives
$$
\sum_{(a,b)\in\Lambda}
\bigl(e^{\overline a b}e^{D_{a,b}}C_{a,b}
-C_{a,b}\bigr)(z,w)e^{\overline a z+bw}=0.
$$
Since the pairs in $\Lambda$ are distinct, the corresponding
exponentials are linearly independent over $\C[z,w]$ by the argument
in the proof of
\cite[Theorem 3.5]{Bauer2015}. Comparing coefficients gives
\begin{equation}\label{eq:block-equation}
e^{\overline a b}e^{D_{a,b}}C_{a,b}=C_{a,b}
\end{equation}
for $(a,b)\in\Lambda.$

Fix $(a,b)\in\Lambda$ with $C_{a,b}\ne0$. Set
$m=\deg C_{a,b}$ and $D=D_{a,b}$.
By \eqref{E3}, the terms of degree $m$ in $e^DC_{a,b}$
are the same as those in $C_{a,b}$, since every positive power of $D$
lowers total degree. Comparing the terms of degree $m$ in
\eqref{eq:block-equation}, we find $e^{\overline a b}=1$.
Combining this equality with \eqref{eq:block-equation} gives
\begin{equation}\label{E4}
\overline a b\in2\pi i\mathbb Z\quad\text{and}\quad
 e^DC_{a,b}=C_{a,b}.
\end{equation}
To reduce the second relation in \eqref{E4} to an equation involving
$D$, define on $V_m$
$$
N_D=\sum_{\nu=1}^{m}\frac{D^\nu}{(\nu+1)!}.
$$
The sum is zero when $m=0$. Since $D^{m+1}=0$ on $V_m$
by \eqref{E2}, expanding the operator exponential in \eqref{E3} gives
$$
e^D-I=D(I+N_D)=(I+N_D)D
$$
on $V_m$.
The second equality holds because $N_D$ is a polynomial in $D$.
Each term of $N_D$ contains a factor $D$, so $N_D^{m+1}=0$ on $V_m$.
Applying
\cite[Section 8A]{Axler2024} to $N_D$, we obtain
$$
(I+N_D)^{-1}=\sum_{j=0}^{m}(-N_D)^j.
$$
For $C_{a,b}$, this factorization takes the form
$$
e^DC_{a,b}-C_{a,b}=(I+N_D)DC_{a,b}.
$$
Since $I+N_D$ is invertible, we obtain
\begin{equation}\label{eq:nilpotent-condition}
e^DC_{a,b}=C_{a,b}
\quad\Longleftrightarrow\quad
DC_{a,b}=0.
\end{equation}
Combining \eqref{E4} and \eqref{eq:nilpotent-condition} gives
\begin{equation}\label{E5}
D_{a,b}C_{a,b}=0
\end{equation}
for $(a,b)\in\Lambda.$
The latter equation also holds when $C_{a,b}=0$.

If $ab=0$, the product rule and \eqref{E13}, together with
\eqref{E1}, give
\begin{align*}
\partial_z\partial_{\overline z}\bigl[C_{a,b}(z,\overline z)K_{a,b}(z)\bigr]
&=K_{a,b}(z)
\bigl(D_{a,b}C_{a,b}+\overline a b\,C_{a,b}\bigr)(z,\overline z)\\
&=K_{a,b}(z)\bigl(D_{a,b}C_{a,b}\bigr)(z,\overline z)=0,
\end{align*}
where the last two equalities use $\overline a b=0$ and
\eqref{E5}, respectively.
This term is harmonic, as required in the first case of
\eqref{eq:coefficient-condition}.

It remains to consider $(a,b)\in\Lambda$ with $ab\ne0$ and
$R=C_{a,b}\ne0$. If $R$ is constant, the second case in
\eqref{eq:coefficient-condition} holds with $d_{a,b}=0$.
Otherwise, let $m=\deg R\ge1$ and denote the homogeneous part
of degree $m$ by $R_m$. By \eqref{E5}, $D_{a,b}R=0$.
To compare the terms of degree $m-1$, use \eqref{E1} to write
$$
D_{a,b}R=(b\partial_z+\overline a\partial_w)R_m
+\partial_z\partial_wR_m+D_{a,b}(R-R_m).
$$
The last two terms have total degree at most $m-2$ when nonzero.  In particular, the last two terms both vanish when $m=1$.
Comparing the terms of degree $m-1$ in $D_{a,b}R=0$ yields
\begin{equation}\label{E6}
(b\partial_z+\overline a\partial_w)R_m=0.
\end{equation}
Since $R_m$ is homogeneous of degree $m$, write
$$
R_m(z,w)=\sum_{j=0}^{m}\lambda_jz^jw^{m-j},
$$
where $\lambda_0,\ldots,\lambda_m\in\C$. Substituting this
expression into \eqref{E6} gives
$$
\sum_{j=0}^{m-1}
\bigl[b(j+1)\lambda_{j+1}
+\overline a(m-j)\lambda_j\bigr]
z^jw^{m-j-1}=0.
$$
Comparing coefficients and using $b\ne0$, we obtain
$$
\lambda_{j+1}
=-\frac{\overline a}{b}\frac{m-j}{j+1}\lambda_j,
\quad 0\le j\le m-1.
$$
For $1\le j\le m$, iterating the recurrence yields
$$
\begin{aligned}
\lambda_j
&=\lambda_0\left(-\frac{\overline a}{b}\right)^j
  \prod_{\ell=0}^{j-1}\frac{m-\ell}{\ell+1}=\lambda_0\binom mj
  \left(-\frac{\overline a}{b}\right)^j.
\end{aligned}
$$
If $\lambda_0=0$, the recurrence would give $\lambda_j=0$
for every $j=0,\ldots,m$, contrary to $R_m\not\equiv0$.
Thus $\lambda_0\ne0$. By the binomial formula,
$$
\begin{aligned}
R_m(z,w)
&=\lambda_0\sum_{j=0}^{m}\binom mj
  \left(-\frac{\overline a}{b}z\right)^jw^{m-j}\\
&=\lambda_0\left(w-\frac{\overline a}{b}z\right)^m.
\end{aligned}
$$
With $\gamma=\lambda_0/(-b)^m\in\C\setminus\{0\}$, we obtain
\begin{equation}\label{E7}
R_m(z,w)=\gamma(\overline a z-bw)^m.
\end{equation}

To bound $m$, write
\begin{equation}\label{E8}
R(z,w)=\sum_{r,s=0}^{m}t_{rs}z^rw^s,
\qquad M=(t_{rs})_{r,s=0}^{m}.
\end{equation}
For $0\le r,s\le m$ with $r+s\ge m$, the equality $\deg R=m$
and \eqref{E7} give
\begin{equation}\label{E9}
t_{rs}=
\begin{cases}
0, & r+s>m,\\[2pt]
\displaystyle
\gamma\binom mr\overline a^{\,r}(-b)^{m-r}\ne0,
& r+s=m.
\end{cases}
\end{equation}
The second expression is nonzero because $\gamma\ne0$ and $ab\ne0$.
Let $\widehat M=(t_{r,m-s})_{r,s=0}^{m}$ be obtained by reversing
the columns of $M$. The subscript $rs$ denotes the entry with
row and column indices $r,s$. Since
$$
r>s\ \Longrightarrow\ r+m-s>m,
$$
equation \eqref{E9} gives
$$
(\widehat M)_{rs}=t_{r,m-s}=0
\qquad(r>s).
$$
Thus
$$
\widehat M=
\begin{pmatrix}
t_{0,m} & t_{0,m-1} & \cdots & t_{0,0}\\
0       & t_{1,m-1} & \cdots & t_{1,0}\\
\vdots  & \ddots    & \ddots & \vdots\\
0       & \cdots    & 0      & t_{m,0}
\end{pmatrix},
$$
and
$$
\det\widehat M
=\prod_{r=0}^{m}t_{r,m-r}
=\prod_{r=0}^{m}
\left[\gamma\binom mr\overline a^{\,r}(-b)^{m-r}\right]
\ne0.
$$
Reversing columns preserves rank, so
\begin{equation}\label{E10}
\operatorname{rank}M
=\operatorname{rank}\widehat M
=m+1.
\end{equation}

We next bound $\operatorname{rank}M$ using the representation of
$R=C_{a,b}$ in \eqref{eq:merged-coefficients}.
Since the parameters within each list are distinct, the relations
$a_{j_0}=a$ and $b_{\ell_0}=b$ determine at most one pair
$(j_0,\ell_0)$. The same holds for $c_{m_0}=a$ and $d_{n_0}=b$.
Each sum in \eqref{eq:merged-coefficients} contains at most one term,
so we may write
\begin{equation}\label{E11}
R(z,w)=P(z)Q(w)-U(z)V(w),
\end{equation}
where
$$
(P,Q)=(p_{j_0},q_{\ell_0}^*)\quad\text{and}\quad
(U,V)=(r_{m_0},s_{n_0}^*)
$$
when the corresponding index pairs exist. If either sum is
empty, we take its two factors to be zero. Write
$$
\begin{aligned}
P(z)&=\sum_{r=0}^{N_P}u_rz^r,
&Q(w)&=\sum_{s=0}^{N_Q}v_sw^s,\\
U(z)&=\sum_{\mu=0}^{N_U}x_\mu z^\mu,
&V(w)&=\sum_{\nu=0}^{N_V}y_\nu w^\nu,
\end{aligned}
$$
where $N_P,N_Q,N_U,N_V$ are nonnegative integers.
We take coefficients beyond the corresponding bounds to be zero.
Expanding the products in \eqref{E11}, we obtain
$$
\begin{aligned}
P(z)Q(w)-U(z)V(w)
&=\sum_{r=0}^{N_P}\sum_{s=0}^{N_Q}
u_rv_s z^rw^s\\
&\quad-\sum_{\mu=0}^{N_U}\sum_{\nu=0}^{N_V}
x_\mu y_\nu z^\mu w^\nu.
\end{aligned}
$$
Comparing the coefficients of $z^rw^s$ with \eqref{E8} gives
\begin{equation}\label{E12}
t_{rs}=u_rv_s-x_ry_s
\end{equation} 
for $0\le r,s\le m$. To write this relation in matrix
form, set
$$
\begin{aligned}
\mathbf u&=(u_0,\ldots,u_m)^{\mathrm T},
&\mathbf v&=(v_0,\ldots,v_m)^{\mathrm T},\\
\mathbf x&=(x_0,\ldots,x_m)^{\mathrm T},
&\mathbf y&=(y_0,\ldots,y_m)^{\mathrm T},
\end{aligned}
$$
where ${}^{\mathrm T}$ denotes transpose.
By \eqref{E12}, for $0\le r,s\le m$ we have
$$
\bigl(\mathbf u\mathbf v^{\mathrm T}
-\mathbf x\mathbf y^{\mathrm T}\bigr)_{rs}
=u_rv_s-x_ry_s=t_{rs}.
$$
Equality of the corresponding entries shows that
$$
M=\mathbf u\mathbf v^{\mathrm T}
-\mathbf x\mathbf y^{\mathrm T}.
$$
In terms of its columns, $M$ takes the form
$$
\begin{aligned}
M
&=
\begin{pmatrix}
u_0v_0-x_0y_0 & \cdots & u_0v_m-x_0y_m\\
\vdots & \ddots & \vdots\\
u_mv_0-x_my_0 & \cdots & u_mv_m-x_my_m
\end{pmatrix}\\
&=
\begin{pmatrix}
v_0\mathbf u-y_0\mathbf x
& \cdots &
v_m\mathbf u-y_m\mathbf x
\end{pmatrix}.
\end{aligned}
$$
Every column of $M$ belongs to
$\operatorname{span}\{\mathbf u,\mathbf x\}$. By \eqref{E10},
$$
m+1=\operatorname{rank}M
\le \dim\operatorname{span}\{\mathbf u,\mathbf x\}
\le2.
$$
This implies $\deg R=m\le1$, and we may write
$$
R(z,w)=c+uz+vw,
$$
where $c,u,v\in\C$ are constants. By equations \eqref{E1} and
\eqref{E5},
$$
0=D_{a,b}R=bu+\overline a v.
$$
Since $a\ne0$, the relation $bu+\overline a v=0$ implies
$v=-bu/\overline a$. Substituting this expression for $v$, we obtain
$$
R(z,w)=c+\frac{u}{\overline a}(\overline a z-bw)
=c+d(\overline a z-bw)
\qquad\text{for some }d\in\C.
$$
We now set $w=\overline z$ and combine this identity with
$\overline a b\in2\pi i\mathbb Z$ from \eqref{E4} to obtain the
second case of \eqref{eq:coefficient-condition}.
\end{proof}

The next example gives a nonconstant coefficient satisfying
\eqref{eq:coefficient-condition}.
\begin{remark}\label{rem:affine-example}
Set
$$
f(z)=ze^z,\quad h(z)=e^z,\quad
g(z)=e^{-2\pi iz},\quad
k(z)=-2\pi iz\,e^{-2\pi iz}.
$$
These four functions belong to $\mathcal A_1$, and their products satisfy
$$
f\overline g-h\overline k
=(z-2\pi i\overline z)e^{z+2\pi i\overline z}.
$$
By \eqref{eq:merged-coefficients}, the only nonzero coefficient is
$$
C_{1,2\pi i}(z,w)=z-2\pi iw.
$$
For $(a,b)=(1,2\pi i)$, we have
$$
\overline a b=2\pi i\quad\text{and}\quad
C_{1,2\pi i}(z,\overline z)
=z-2\pi i\overline z
=\overline a z-b\overline z.
$$
This verifies \eqref{eq:coefficient-condition} with
$c_{a,b}=0$ and $d_{a,b}=1$.
Proposition \ref{prop:coefficient-criterion} and
Lemma \ref{lem:berezin-injectivity} (b) show that
$T_{f+\overline k}$ and $T_{h+\overline g}$ commute.
Applying \eqref{eq:polynomial-block-transform} to the two products
at $(a,b)=(1,2\pi i)$ gives
$$
\mathcal B[f\overline g]-f\overline g
=\mathcal B[h\overline k]-h\overline k
=2\pi i e^{z+2\pi i\overline z}\ne0.
$$
By Lemma \ref{lem:berezin-injectivity} (a), this nonzero value
excludes condition (c)(i) of Theorem \ref{thm:bauer-characterization}.
The families $\{f,h,1\}$ and $\{g,k,1\}$ are linearly independent,
so conditions (c)(iii) and (c)(iv) also fail.
Finally, a relation
$a(f+\overline k)+b(h+\overline g)+c=0$
gives, upon differentiation with respect to $z$,
$$
\bigl[a(z+1)+b\bigr]e^z=0.
$$
Hence $a=b=c=0$, which excludes (c)(ii).
\end{remark}

\section{The corrected characterization}

Throughout this section, we assume that
$f,g,h,k\in\Hol(\C)\cap\mathcal E(\C)$.
We use the three cases in \cite[Section 4]{Bauer2015}.
In case (A), all of $f,g,h,k$ belong to $\mathcal A_1$.
In case (S), some but not all of them belong to this algebra.
In case (N), none of them does.
The results in \cite[Lemmas 4.6--4.8 and Proposition 4.9]{Bauer2015}
do not depend on Proposition 4.5.

Combining Proposition \ref{prop:coefficient-criterion}
with the results for cases (S) and (N) gives the following
characterization of fixed points of the Berezin transform.
\begin{theorem}[Correction of {\cite[Theorem 4.14]{Bauer2015}}]
\label{thm:corrected-fixed-points}
Let $f,g,h,k\in\Hol(\C)\cap\mathcal E(\C)$. Then
$$
\mathcal B[f\overline g-h\overline k]
=f\overline g-h\overline k
$$
if and only if at least one of the following conditions holds.
\begin{enumerate}
\item[(i)] $\mathcal B[f\overline g]=f\overline g$ and
$\mathcal B[h\overline k]=h\overline k$.
\item[(ii)] $f\overline g-h\overline k$ is harmonic.
\item[(iii)] There are constants $a,b,c\in\C$ with $ab\ne0$ such that
$$
af+bh=c\quad\text{and}\quad
\mathcal B[f(b\overline g+a\overline k)]
=f(b\overline g+a\overline k).
$$
\item[(iv)] There are constants $a,b,c\in\C$ with $ab\ne0$ such that
$$
ag+bk=c\quad\text{and}\quad
\mathcal B[(\overline b f+\overline a h)\overline g]
=(\overline b f+\overline a h)\overline g.
$$
\item[(v)] All four functions belong to $\mathcal A_1$, and the
family $C=(C_{a,b})_{(a,b)\in\Lambda}$ in
\eqref{eq:merged-coefficients} satisfies
\eqref{eq:coefficient-condition}.
\end{enumerate}
\end{theorem}

\begin{proof}
Set $W=f\overline g-h\overline k$.
Conditions (i) and (ii) are sufficient because $\mathcal B$
is linear and fixes harmonic functions in $\mathcal E(\C)$.
For (iii), apply \cite[Lemma 4.13]{Bauer2015} to $f,g,-h,k$.
Under (iv), complex conjugation of the second identity gives
$$
\mathcal B[g(b\overline f+a\overline h)]
=g(b\overline f+a\overline h).
$$
The same lemma, applied to $g,f,-k,h$, yields
$\mathcal B[\overline W]=\overline W$.
Taking complex conjugates gives $\mathcal B[W]=W$.
Condition (v) is sufficient by Proposition \ref{prop:coefficient-criterion}.

For necessity, assume $\mathcal B[W]=W$.
If one of the four functions is constant, one product is harmonic
and hence fixed. Linearity implies that the other product is fixed
as well, so (i) holds.
We may now assume that all four functions are nonconstant.

In case (A), Proposition \ref{prop:coefficient-criterion} gives (v).
In case (S), apply \cite[Proposition 4.9 and Lemma 4.12]{Bauer2015}
to $f,g,-h,k$.
The first two alternatives in Lemma 4.12 give (ii) and (iii),
respectively. Its third alternative gives
$$
ag-bk=c,
\qquad
\mathcal B[(\overline b f-\overline a h)\overline g]
=(\overline b f-\overline a h)\overline g
$$
for some $a,b,c\in\C$ with $ab\ne0$.
Replacing $b$ by $-b$ and multiplying the second identity by $-1$
gives (iv).
In case (N), \cite[Proposition 4.10]{Bauer2015},
applied to $f,g,-h,k$, gives (ii). The proof is complete.
\end{proof}

Lemma \ref{lem:berezin-injectivity} (a) converts condition (i) of
Theorem \ref{thm:corrected-fixed-points} into the vanishing of the
two semi-commutators. When all four functions are nonconstant,
\cite[(3.20)]{Bauer2015}, applied to $f,g,-h,k$, identifies its
harmonic alternative with the linear dependence of
$f+\overline k,h+\overline g,1$. If one function is constant,
the harmonic alternative is already covered by (i).
Together with Lemma \ref{lem:berezin-injectivity} (b), these
observations give the following characterization.

\begin{theorem}[Correction of {\cite[Theorem 4.15]{Bauer2015}}]
\label{thm:corrected-commuting}
Let $f,g,h,k\in\Hol(\C)\cap\mathcal E(\C)$. The following
statements are equivalent.
\begin{enumerate}
\item[(a)] $[T_{f+\overline k},T_{h+\overline g}]=0$.
\item[(b)] $\mathcal B[f\overline g-h\overline k]
=f\overline g-h\overline k$.
\item[(c)] At least one of the following conditions holds.
\begin{enumerate}
\item[(i)] $(T_f,T_{\overline g}]=(T_h,T_{\overline k}]=0$.
\item[(ii)] $\{f+\overline k,h+\overline g,1\}$ is linearly dependent.
\item[(iii)] There are constants $a,b,c\in\C$ with $ab\ne0$ such that
$$
af+bh=c\quad\text{and}\quad
\mathcal B[f(b\overline g+a\overline k)]
=f(b\overline g+a\overline k).
$$
\item[(iv)] There are constants $a,b,c\in\C$ with $ab\ne0$ such that
$$
ag+bk=c\quad\text{and}\quad
\mathcal B[(\overline b f+\overline a h)\overline g]
=(\overline b f+\overline a h)\overline g.
$$
\item[(v)] All four functions belong to $\mathcal A_1$, and the
family $C=(C_{a,b})_{(a,b)\in\Lambda}$ in
\eqref{eq:merged-coefficients} satisfies
\eqref{eq:coefficient-condition}.
\end{enumerate}
\end{enumerate}
\end{theorem}

Two or more of the original four conditions may hold
at the same time. In case (A),
Proposition \ref{prop:coefficient-criterion} reduces the problem
to checking finitely many conditions on the coefficients.
In cases (S) and (N), the original four alternatives still
cover all possibilities. For the example in \eqref{eq:symbols}, the terms at
$(a,b)=(1,\pi i)$ cancel. The only nonzero coefficients are $C_{1,2\pi i}=1$ and
$C_{2,\pi i}=-1$. Both parameter pairs satisfy $\overline a b=2\pi i$,
so these constant coefficients satisfy
\eqref{eq:coefficient-condition} with $d_{a,b}=0$.


\begin{thebibliography}{9}
\bibitem{Axler2024}
S. Axler,
\emph{Linear Algebra Done Right},
4th ed., Undergraduate Texts in Mathematics,
Springer, Cham, 2024.
\href{https://doi.org/10.1007/978-3-031-41026-0}
{doi:10.1007/978-3-031-41026-0}.
\bibitem{Bauer2015}
W. Bauer, B. R. Choe, and H. Koo,
\emph{Commuting Toeplitz operators with pluriharmonic symbols on the Fock
space},
J. Funct. Anal. \textbf{268} (2015), 3017--3060.
\href{https://doi.org/10.1016/j.jfa.2015.03.003}{doi:10.1016/j.jfa.2015.03.003}.

\end{thebibliography}
\end{document}